\documentclass[a4paper,11pt,onecolumn]{amsart}
\usepackage{amsfonts}
\usepackage{pb-diagram}
\title[Homology and entropy]{Entropy of automorphisms, homology and the intrinsic polynomial structure of nilpotent groups}

\newtheorem{thm}{Theorem}
\newtheorem{lemma}[thm]{Lemma}

\newtheorem{cor}[thm]{Corollary}
\newtheorem{prop}[thm]{Proposition}

\newtheorem{quest}[thm]{Question}

\numberwithin{thm}{section}

\newcommand{\al}{\alpha}
\newcommand{\eps}{\epsilon}

\newcommand{\bZ}{\mathbb{Z}}
\newcommand{\bR}{\mathbb{R}}
\newcommand{\bC}{\mathbb{C}}

\DeclareMathOperator{\Aut}{Aut}
\DeclareMathOperator{\Out}{Out}

\DeclareMathOperator{\rk}{rk}

\newcommand{\bQ}{\mathbb{Q}}

\author[T. Koberda]{Thomas Koberda}
\address{Department of Mathematics\\ Harvard University\\ 1 Oxford St.\\ Cambridge, MA 02138 }
\email{ koberda@math.harvard.edu}
\begin{document}
\begin{abstract}
We study the word length entropy of automorphisms of residually nilpotent groups, and how the entropy of such group automorphisms relates to the entropy of induced automorphisms on various nilpotent quotients.  We show that much like the structure of a nilpotent group is dictated to a large degree by its abelianization, the entropy of an automorphism of a nilpotent group is dictated by its entropy on the abelianization.  We give some applications to the study of pseudo-Anosov homeomorphisms of surfaces.  In particular, we show that if $\psi$ is a non--homological pseudo-Anosov homeomorphism of a surface $\Sigma$ with dilatation $K$ and $N$ is any nilpotent quotient of any finite index characteristic subgroup of $\pi_1(\Sigma)$ to which $\psi$ descends, the entropy of $\psi$ viewed as an automorphism of $N$ is bounded away from $K$.  This answers a question of D. Sullivan.
\end{abstract}
\maketitle
\begin{center}
\today
\end{center}
\tableofcontents
\section{Introduction}
Let $G$ be a finitely generated group and $\phi\in\Aut(G)$.  The goal of this paper is to understand some of the asymptotic properties of $\phi$ which can be gleaned from considering certain quotients of $G$ to which the automorphism $\phi$ descends.  In particular, we will be interested in the relationship between quotients of subgroups of $G$ and the entropy of $\phi$.

The underlying philosophy behind this paper is as follows: it is well--known that the structure of finitely generated nilpotent groups is essentially polynomial in nature.  It can be easily deduced from commutator identities that virtually nilpotent groups have polynomial word growth.  Conversely, Gromov showed that groups with polynomial growth are virtually nilpotent (see \cite{Gr}, also \cite{Kl}).  In \cite{Ka}, Karidi described the intrinsically polynomial structure of balls inside of real simply connected nilpotent Lie groups.  We will establish the notion that the automorphisms of a nilpotent group inherit the polynomial--like properties of the underlying group.

We first recall the notion of entropy: let $S$ be a finite generating set for $G$.  There is a canonical (up to quasi-isometry) length function $\ell=\ell_S$ on $G$, which we call the {\bf word length}.  By definition, $\ell(g)$ is the function $d(g,1)$, where $d$ denotes the graph metric in the Cayley graph of $G$ with respect to the generating set $S$.  The {\bf entropy} $K_{\phi}$ is defined by \[\max_{s\in S}\lim_{n\to\infty}\ell(\phi^n(s))^{1/n}.\]  Since $\ell(\phi^n(s))$ is submultiplicative in $n$, the limit exists.  It is standard that $K_{\phi}$ is independent of $S$ (see \cite{FLP}).  We note that our definition of entropy differs slightly from the standard one.  Usually, the entropy of an automorphism $h(\phi)$ is given by $\log K_{\phi}$.  We say that $\phi$ has trivial entropy when $K_{\phi}=1$, which in the standard terminology means that $\phi$ has no entropy.

It is easy to check that the entropy satisfies $K_{\phi}\geq 1$, and that conjugation by an element of $G$ has trivial entropy.  It is slightly less trivial to see that the entropy of an automorphism depends only on its image in $\Out(G)$.  It follows that if $\Out(G)$ is finite then $G$ admits no automorphisms with nontrivial entropy.  For instance, Mostow--Prasad rigidity implies that lattices in $PSL_2(\bC)$ admit no positive entropy automorphisms (see \cite{BP}).

The main motivation for the study entropy in this paper comes from the theory of {\bf pseudo-Anosov homeomorphisms}.  Recall that a homeomorphism $\psi$ of a surface $\Sigma$ of genus $g\geq 1$ and $n\geq 0$ punctures has a Nielsen--Thurston classification.  This classification says that $\psi$ is either finite order, reducible, or pseudo-Anosov.  In the first case, $\psi$ has a power which is isotopic to the identity.  In the second case, there is a finite nonempty collection $\mathcal{C}$ of disjoint, non--peripheral, essential simple closed curves which is preserved by $\psi$, up to isotopy.  In the third case, there is a pair of measured foliations which are preserved by a representative of $\psi$.  The foliations together with $\psi$ have a canonical real algebraic integer $K_{\psi}$ associated to them, called the {\bf dilatation} of $\psi$.  This is the factor by which the stable and unstable foliations are stretched and contracted, respectively.

Since a mapping class can be regarded as an outer automorphism of the fundamental group of the surface, we can talk about the entropy of $\psi$ as an outer automorphism.  It turns out that the entropy of a pseudo-Anosov homeomorphism coincides with its dilatation (see \cite{FLP}).  A rough reason for this is that the stable and unstable foliations determine a piecewise Euclidean metric $\ell_{euc}$ on $\Sigma$ which asymptotically agrees with any hyperbolic metric on $\Sigma$, in the sense that for any sufficiently long word $g\in\pi_1(\Sigma)$ represented by a hyperbolic geodesic $\gamma$, we have \[\ell_{euc}(\gamma)\sim\ell_{hyp}(\gamma)\sim\ell_S(g).\]

To compute the dilatation of a pseudo-Anosov homeomorphism and to show that it is an algebraic integer, one associates a canonical characteristic class to the stable foliation of $\psi$, which turns out to be an eigenclass for the action of $\psi$ on the real homology of the surface.  When the stable foliation is orientable, then the action of $\psi$ on $H_1(\Sigma,\bR)$ has a simple maximal eigenvalue equal to $K_{\psi}$, and this number is an algebraic integer since it is a root of the characteristic polynomial of a matrix with integer entries (see \cite{FLP}, \cite{Kob}, \cite{Mc1}, \cite{Th}).  In this case, we call $\psi$ a {\bf homological pseudo-Anosov}.  When the foliation is not orientable, there is a double cover which is branched over the odd--order singularities of the foliation such that the foliation lifts to an orientable one.  Furthermore, $\psi$ lifts to this double cover, preserving the folitation.  Typically, pseudo-Anosov homeomorphisms do not have even locally orientable stable and unstable foliations.  A non--locally orientable foliation will remain non--orientable on each cover of the base surface.  We refer to such pseudo-Anosov homeomorphisms as {\bf non--homological}, since $K_{\psi}$ will never be an eigenvalue of the action of $\psi$ on the homology of any cover.

On an unbranched cover of $\Sigma'\to \Sigma$, the simplicity and maximality of $K_{\psi}$ as mentioned above implies that any other eigenvalue of the action of $\psi$ on $H_1(\Sigma',\bR)$ is strictly smaller than $K_{\psi}$.  In particular, if the stable foliation is not orientable, the spectral radius of the action of $\psi$ on $H_1(\Sigma',\bR)$ is strictly smaller than $K_{\psi}$.  This follows from the fact that the eigenclass corresponding to $K_{\psi}$ is in the kernel of the map on real homology induced by the double branched orientation cover (see \cite{KobSil}, \cite{LT}).  We remark that if $\Sigma'\to\Sigma$ is a finite cover to which a pseudo-Anosov homeomorphism lifts, the lifted homeomorphism will still be pseudo-Anosov with the same dilatation (see \cite{Kob}).  Furthermore, the spectral radius of the action of $\psi$ on $H_1(\Sigma',\bR)$ is bounded above by $K_{\psi}$, where $\Sigma'$ is any finite $\psi$--invariant cover of $\Sigma$.  This follows from general facts about Lipschitz homeomorphisms acting on the homology of compact manifolds.  A precise statement and commentary can be found in \cite{Kob}.

General principles imply that if $G$ is a finitely generated group and $\phi$ is an automorphism with entropy $K_{\phi}$, and if $\overline{G}$ is a quotient to which $\phi$ descends to an automorphism $\overline{\phi}$, we obtain the inequality $K_{\overline{\phi}}\leq K_{\phi}$.  This follows from the fact that taking word length is non--increasing under taking quotients.  One might be tempted to guess that $K_{\psi}$ can be obtained as the supremum of such homological spectral radii, as $\Sigma'$ varies over all $\psi$--invariant finite covers of $\Sigma$.  However, we have the following result of McMullen (cf. \cite{Kob}):

\begin{thm}[\cite{Mc2}]
Suppose that $\psi$ is a pseudo-Anosov homeomorphism of a surface with dilatation $K$.  Then either $K$ is the spectral radius of the action of $\psi$ on a finite cover of $\Sigma$ and the stable foliation of $\psi$ becomes orientable on that cover, or there is an $0\leq\al<1$ such that the spectral radius of the action of $\psi$ on $H_1(\Sigma',\bR)$ is no larger than $\al\cdot K$, where $\Sigma'$ is any finite $\psi$--invariant cover of $\Sigma$.
\end{thm}

McMullen's theorem says that one cannot find the entropy of a pseudo-Anosov homeomorphism by taking finite index subgroups of $\pi_1(\Sigma)$ and abelianizing.  One is thus naturally led to larger quotients of $\pi_1(\Sigma)$.  The simplest infinite nonabelian quotients of $\pi_1(\Sigma)$ are nilpotent quotients.  Furthermore, if $\psi$ is a nontrivial automorphism of $\pi_1(\Sigma)$ then its nontriviality is visible on a nilpotent quotient of $\pi_1(\Sigma)$, since $\pi_1(\Sigma)$ is residually nilpotent.

Nilpotent quotients of groups have factored significantly into the understanding of the rational homotopy type of manifolds, as expounded by many authors (see \cite{S}, for instance).  D. Sullivan has asked (\cite{SullMc}) whether the entropy of a pseudo-Anosov homeomorphism on sequences of larger and larger nilpotent quotients of $\pi_1(\Sigma)$ (or perhaps of finite index subgroups of $\pi_1(\Sigma)$) converges to the dilatation of $\psi$.

The primary results in this paper implies that homology of finite index subgroups of $\pi_1(\Sigma)$ divulges just as much information about entropy of automorphisms as arbitrary nilpotent quotients of finite index subgroups.  We thus answer Sullivan's question in the negative:

\begin{thm}\label{t:pos}
Let $\phi\in\Aut(N)$ and let $\overline{\phi}$ be the induced automorphism of $H_1(N,\bZ)$.  Then $K_{\phi}=K_{\overline{\phi}}$.
\end{thm}

\begin{cor}
Let $\psi$ be a pseudo-Anosov homeomorphism with a non-orientable stable foliation.  Then $K_{\psi}$ cannot be detected from the induced entropy on nilpotent quotients of finite index subgroups of the fundamental group.
\end{cor}

We say that an automorphism of a group $G$ is {\bf homologically trivial} if it induces the identity on $H_1(G,\bZ)$.

\begin{cor}\label{t:trivial}
Let $\phi$ be a homologically trivial automorphism of a finitely generated nilpotent group $N$.  Then $K_{\phi}=1$.
\end{cor}

More generally:

\begin{cor}\label{t:uni}
If $\phi\in\Aut(N)$ induces a unipotent automorphism of $H_1(N,\bQ)$ then $K_{\phi}=1$.
\end{cor}

Thus, if an automorphism of a nilpotent group has trivial homological entropy, then this fact is reflected by an absence of entropy for the automorphism of the whole group.  We will in fact show that if $\phi$ is an automorphism of a nilpotent group $N$ which acts unipotently on the homology of $N$ then the word growth under iterations of $\phi$ is bounded by a polynomial.

\begin{thm}\label{t:unipoly}
Let $N$ be a finitely generated nilpotent group, let $g\in N$ and let $\phi$ be an automorphism of $N$ which induces a unipotent automorphism of $H_1(N,\bQ)$.  Then for each $g\in N$, \[\ell(\phi^n(g))=O(n^r)\] for some $r>1$ which is independent of $g$.
\end{thm}
It does not seem possible to determine $r$ simply from the data of the action of $\phi$ on $H_1(N,\bQ)$.

Notice that Theorem \ref{t:unipoly} is stronger than and implies both Corollaries \ref{t:trivial} and \ref{t:uni}.

An alternative perspective on Theorem \ref{t:pos} can be formulated using the (incorrect) notion that if an automorphism $\phi$ of a residually nilpotent group $G$ has nontrivial entropy, then one should be able to observe nontrivial entropy for the action of $\phi$ on some nilpotent quotient of $G$.  To make the notion precise, write the lower central series of $G$ by $\{\gamma_i(G)\}$, so that $\gamma_1(G)=G$ and $\gamma_i(G)=[\gamma_{i-1}(G),G]$.  We call the quotient $G/\gamma_i(G)$ the {\bf $i^{th}$ universal nilpotent quotient of $G$}.  We see that \[K_{\phi}=\left (\lim_{n\to\infty}\lim_{i\to\infty}\max_{s\in S}\ell(\phi_i^n (s_i))\right )^{1/n},\] where $S$ is a finite generating set for $G$, $s_i$ is the image of $s$ in $G/\gamma_i(G)$, and $\phi_i$ is the automorphism of $G/\gamma_i(G)$ induced by $\phi$.  One might suppose that the two limits which express $K_{\phi}$ should commute with each other.  Theorem \ref{t:pos} shows that in general, they do not.

\section{Acknowledgements}
The author thanks J. Ellenberg, B. Farb, C. McMullen and A. Suciu for useful comments and discussions, and he thanks the anonymous referee for many useful comments, corrections and suggestions.

\section{Preliminary remarks and results}
We first need to collect some useful observations that will assist in later proofs and in motivation.  As a matter of notation and terminology, we will say that a sequence of positive real numbers $\{c_n\}$ grows like $K^n$ for $K>0$ and write $c_n\sim K^n$ if there exists a $C\geq 1$ such that for every $\eps>0$ there exists an $N$ such that for all $n\geq N$, \[\frac{1}{C}(K-\eps)^n\leq c_n\leq C(K+\eps)^n.\]  We will use the symbols $\succeq$ and $\preceq$ if only one of the inequalities holds.  We use the ``big $O$" and ``little $o$" notations in the standard way: $f(n)=O(g(n))$ if there is a $C>0$ such that $f(n)\leq C\cdot g(n)$ and $f(n)=o(g(n))$ if \[\frac{f(n)}{g(n)}\to 0\] as $n$ tends to infinity.

A standing assumption throughout this paper is that all nilpotent groups in question are torsion-free.  This is a reasonable assumption for the following two reasons:

\begin{prop}\label{p:vtf}
Let $N$ be a finitely generated nilpotent group.  Then $N$ is virtually torsion-free.
\end{prop}
\begin{proof}
This follows from the linearity of torsion-free nilpotent groups and Selberg's Lemma (see \cite{R}).
\end{proof}

\begin{prop}\label{p:fi}
Let $G$ be a finitely generated group, $\phi\in\Aut(G)$ and $G'<G$ a finite index $\phi$--invariant subgroup.  The entropy of $\phi$ as an automorphism of $G$ coincides with its entropy as an automorphism of $G'$.
\end{prop}
\begin{proof}
In one direction, the inclusion map $G'\to G$ is a quasi-isometry.  If $g'\in G'$ and $\ell(\phi^n(g'))\sim K^n$ then $\ell(\phi^n(g'))\sim K^n$ within $G$.  It follows that if the entropy of $\phi$ on $G'$ is $K$ then it is at least $K$ on $G$.

Conversely, let $g\in G$ be such that $\ell(\phi^n(g))\sim K^n$.  Since $G'$ has finite index in $G$, there is a finite collection of coset representatives $\{t_1,\ldots,t_n\}$ for $G'$ in $G$.  By passing to a finite power of $\phi$, we may assume that $\phi$ preserves the cosets setwise.

If $g\in G'$ satisfies $\ell(\phi^n(g))\sim K^n$ then there is nothing to show.  We will show that all the cosets have the same growth rate.  Indeed, we claim that if $G$ is a group, $\phi\in \Aut(G)$, $H, tH\subset G$ are $\phi$--invariant cosets, then the word growth entropy of $\phi$ on $H$ and $tH$ is equal.

Suppose $\phi$ has entropy $K$ when restricted to a coset $H$, and let $\Phi(h)=t\phi(h)$.  Note that $\Phi$ satisfies a cocycle--like identity: \[\Phi^n(h)=t\phi(t)\cdots\phi^{n-1}(t)\phi^n(h).\]  We have the estimate \[\ell(\Phi^n(h))\sim (1+K+\cdots+K^n)\sim K^n.\]  Thus, $\Phi$ has entropy $K$ as well.  The claim follows, hence the proposition.
\end{proof}

When $G'<G$ is not of finite index, all bets are off.  For example, there are automorphisms of the Heisenberg group with nontrivial entropy which act trivially on the center.  It is also easy to produce a two--step nilpotent group $N$ with an automorphism $\phi$ which has spectral radius $K>1$ on $N^{ab}$ and spectral radius greater than $K$ on $Z(N)$.

One might be inclined to conjecture that Proposition \ref{p:fi} holds for quasi-isometrically embedded subgroups which are $\phi$--invariant, but this is not the case -- consider for instance an automorphism $\phi$ of $G$ and view it as an automorphism of $G\times\bZ$.  We extend $\phi$ by declaring it to act trivially on the cyclic direct factor.  The cyclic factor is quasi--isometrically embedded but $\phi$ acts trivially on it, no matter how $\phi$ acts on $G$.

The other fundamental tool that we shall use in this paper is the geometry of balls inside of nilpotent groups.  Recall that nilpotent groups have polynomial distortion of subgroups.  What this means is that for any group $N$, the quotients $\gamma_i(N)/\gamma_{i+1}(N)$ are not quasi--isometrically embedded in $N/\gamma_{i+1}(N)$, but rather have polynomial distortion whose degree depends on $i$.  This phenomenon follows fundamentally from certain commutator identities which hold in nilpotent groups (see \cite{MKS}).  One precise formulation was proved by Osin in \cite{O}:

\begin{lemma}\label{l:osin}
Let $N\to N'$ be an inclusion of finitely generated nilpotent groups, with length functions $\ell_N$ and $\ell_N'$, respectively.  Then there is a polynomial $P$ depending on $N$ and $N'$ such that if $n\in N$ satisfies $\ell_{N'}(n)=k$, then $\ell_N(n)\leq P(k)$.  In particular inclusions of nilpotent groups have polynomial word distortion.
\end{lemma}

Osin's result has the following easy but useful corollary:

\begin{cor}
Let $N$ be a finitely generated nilpotent group, let $\phi$ be an automorphism of $N$, and let $N'$ be a $\phi$--invariant subgroup of $N$.  Suppose that the entropy of $\phi$ as an automorphism of $N'$ is $K'$.  Then $K_{\phi}\geq (K')^{1/k}$, where $k$ is a positive integer which depends only on the inclusion $N'\to N$.  In particular $K_{\phi}=1$ implies $K'=1$.
\end{cor}
\begin{proof}
Let $g\in N'$ be such that $\ell_{N'}(\phi^n(g))\sim (K')^n$.  Since $N'$ is polynomially distorted in $N$, there is a $k$ such that $\ell_N(\phi^n(g))\succeq(K')^{n/k}$, so that $K_{\phi}\geq (K')^{1/k}$.
\end{proof}

In addition to Osin's result, we will apply the following theorem of Karidi (see \cite{Ka}), which determines the approximate distortion of balls inside of nilpotent groups:

\begin{thm}\label{t:karidi}
Let $G$ be a simply connected finite dimensional real Lie group equipped with a left--invariant Riemannian metric.  Let $G=\gamma_1(G)\supset \cdots \supset \gamma_{c+1}(G)=1$ be the lower central series of $G$. The quotients $\gamma_i/\gamma_{i+1}$ are vector groups of dimensions, say, $d_i$. Then there exist coordinates in $\bR^n$ and a constant $a>1$ such that, for every $r>1$, the ball of radius $r$ (around the origin, i.e., $1_G$) is contained in the (Euclidean) box with sides parallel to the coordinate axes and sizes \[ar,\cdots,ar,\ (ar)^2,\cdots, (ar)^2,\ \cdots,\ (ar)^c,\cdots,(ar)^c\] (where the first group in this list consists of $d_1$ elements, the second $d_2$ elements, etc.), and another box with sizes \[r/a,\cdots,r/a,\ (r/a)^2,\cdots,(r/a)^2,\ \cdots,\ (r/a)^c,\cdots,(r/a)^c\] is contained in this ball.
\end{thm}

Recall that by a result of Mal'cev (see \cite{R}), to any finitely generated torsion--free nilpotent group $N$ one can associate a simply connected real nilpotent Lie group $N\otimes\bR$, called the Mal'cev completion of $N$.  The quotient $(N\otimes\bR)/N$ is compact, since it is an iterated torus bundle.  It follows that $N$ and $N\otimes\bR$ are quasi--isometric to each other.  It follows that Karidi's Theorem applies to finitely generated torsion--free nilpotent groups, with the Riemannian metric replaced by the word metric with respect to some generating set.  As observed in Proposition \ref{p:vtf}, any finitely generated nilpotent group has a torsion--free subgroup of finite index, so each finitely generated nilpotent group is quasi--isometric to some simply connected real nilpotent Lie group.

We will need to understand how much word length can grow under rearrangement of generators within a word.  We will need to develop some understanding in the following setup: let $N$ be an $m$--generated, $k$--step nilpotent group and $\gamma_k(N)$ be the last term in its lower central series.  Let $w\in F_m$ be an arbitrary reduced word in the free group on $m$ generators.  Suppose that the image of $w$ has length $\ell$ in $N$ and length $\ell'$ in $N/\gamma_k(N)$.

\begin{lemma}\label{l:isoperimetric}
The image $g$ of $w$ in $N$ can be written as a product $g'\cdot z_g$ where $z_g\in \gamma_k(N)$ and $g'$ has length $\ell'$.  Furthermore, we may arrange so that the length of $z_g$ in $\gamma_k(N)$ is $O(\ell^{k})$.
\end{lemma}
\begin{proof}
This results from the fact that $k$--step nilpotent groups have polynomial Dehn functions which are polynomials of degree no more than $k+1$.  For a proof of this fact, consult H. Short's notes in \cite{BRS}.

Choose an arbitrary finite presentation for $N$ with generators \[F=\{n_1,\ldots,n_m\}\] and a finite set $R$ of relations.  Descending modulo the last term of the lower central series, these elements also generate $N/\gamma_k(N)$.  We write $\overline{g}$ for the image of $g$ in $N/\gamma_k(N)$.  Since $F$ generates $N/\gamma_k(N)$, we let $w_{\overline{g}}$ be the shortest word in $F$ which represents $\overline{g}$ in $N/\gamma_k(N)$.  On the other hand, $w_{\overline{g}}$ can be viewed as an element of $N$ via the surjection $F_m\to N$.  This element can be viewed as a pullback of $\overline{g}$, and we will denote it by $\tilde{g}$.  Note that $\tilde{g}g^{-1}$ represents the identity in $N/\gamma_k(N)$ by definition.  It follows that $\tilde{g}g^{-1}$ has area at most $O(\ell^{k})$ in our chosen presentation, since $N/\gamma_k(N)$ is $(k-1)$--step nilpotent.

It follows that we only require $O(\ell^{k})$ conjugates of relators to express $\tilde{g}g^{-1}$ as the identity in $N/Z(N)$.  For the finite presentation \[N\cong\langle F\mid R\rangle,\] adjoin to $R$ a finite set of words to get a set of relations $R'$ for which \[N/\gamma_k(N)\cong\langle F\mid R'\rangle.\]  The normal closure of $R'$ in $N$ is precisely $\gamma_k(N)$, and only $O(\ell^{k})$ conjugates of elements of $R'$ are required to express $\tilde{g}g^{-1}$ as the identity in $N/\gamma_k(N)$ by the definition of the Dehn function.  Since $\gamma_k(N)$ is central in $N$, we have that conjugation by $N$ acts trivially on $NC(R')/NC(R)$, where $NC(--)$ denotes the normal closure in $F_m$.  Note that this conjugation action makes sense since by definition, $N\cong F_m/NC(R)$.  The conclusion of the lemma follows.
\end{proof}

Finally, one more tool:

\begin{lemma}\label{l:eps}
Suppose $G$ is a finitely generated group, $\phi\in\Aut(G)$, and suppose that there are constants $C>0$ and $K,k>1$ such that \[\ell(\phi^n(g))\leq C\cdot n^k\cdot K^n\] for $g$ in a finite generating set.  Then the entropy of $\phi$ is no more than $K$.
\end{lemma}
\begin{proof}
Let $\eps>0$.  Consider \[c_n=\frac{Cn^k K^n}{(K+\eps)^n}.\]  The ratio $c_{n+1}/c_n$ is given by \[\frac{(n+1)^k\cdot K}{n^k\cdot (K+\eps)}.\]  Standard calculus shows that this ratio tends to a constant strictly less than one as $n$ tends to infinity, so that in fact the sequence $\{c_n\}$ is summable.  It follows that the entropy of $\phi$ is less than $K+\eps$, whence the claim.
\end{proof}

\section{Automorphisms of nilpotent groups}

Before we finally give a proof of Theorem \ref{t:pos}, the main result of this paper, we make one observation about eigenvalues and entropy:

\begin{lemma}
Let $A$ be a finitely generated, torsion--free abelian group, and let $\phi$ be an automorphism of $A$.  Suppose $\phi$ has an eigenvalue over $\bC$ which is not a root of unity.  Then the entropy of $\phi$ is strictly larger than one.
\end{lemma}
\begin{proof}
All the eigenvalues of $\phi$ are algebraic integers, since $\phi$ is defined over $\bZ$ and the characteristic polynomial of $\phi$ is an integral monic polynomial.  Let $\al$ be an algebraic integer with $|\al|=1$.  It is well--known that $\al$ is a root of unity if and only if all the Galois conjugates of $\al$ have unit length.  It follows that if $\phi$ has an eigenvalue which is not root of unity then it has an eigenvalue off the unit circle.  Since $\phi$ is an automorphism of $A$, the product of its eigenvalues is equal to $\pm 1$, so that $\phi$ has an eigenvalue of length strictly larger than one.
\end{proof}

\begin{proof}[Proof of Theorem \ref{t:pos}]
We will proceed by induction on the length $k$ of the lower central series of $N$.  The case where $k=1$ is trivial.  We write $\gamma_k(N)$ for the $k^{th}$ term of the lower central series of $N$, which is the last nontrivial one.  We have that the word growth entropy on $N/\gamma_k(N)$ is $K_{\overline{\phi}}$, by the inductive hypothesis.  Let $g$ be an element of $N$ and $\overline{g}$ its image in $N/\gamma_k(N)$.  We may rewrite $g$ in $N$ as $g'\cdot z_g$, where $g'$ has length $\ell(\overline{g})$ and $z_g\in \gamma_k(N)$.  We choose a finite generating set for $N$ and perform this rewriting process for each element in the generating set.  We now apply powers of $\phi$ to a candidate element of $N$ for maximal word growth, which we may assume to be one of the generators.

The effect of applying $\phi$ to an element $g$ is to replace each generator $x$ occurring in an expansion for $g$ by a word of length $\ell(\overline{\phi(x)})$, perturbed by an element of $\gamma_k(N)$.  The action of $\phi$ on $\gamma_k(N)$ is linear since $\gamma_k(N)$ is a torsion--free, finitely generated abelian subgroup of $N$, with spectral radius no larger than $(K_{\overline{\phi}})^{k}$.  By Theorem \ref{t:karidi}, there exists an $a>1$ such that balls of radius $(r/a)^{k}$ in $\gamma_k(N)$ are contained in balls of radius $r$ inside of $N$.  Thus, if $z\in \gamma_k(N)$ and $\ell_{\gamma_k(N)}(\phi^n(z))\sim K^{kn}$, we have that $\ell_N(\phi^n(z))\sim K^n$.

Now we can estimate $\ell(\phi^n(g))$.  The length of the image $\overline{\phi^n(g)}$ of $\phi^n(g)$ in $N/\gamma_k(N)$ is $\sim K_{\overline{\phi}}^n$.  Each generator of $N$ occurring in an expansion for $\overline{\phi^n(g)}$ contributes a bounded length word in $\gamma_k(N)$ upon application of $\phi$.  Applying Lemma \ref{l:isoperimetric}, we see that $\phi(g)$ can be written as \[\phi(g)'\cdot z_{\phi(g)},\] where $\ell(\phi(g)')=\ell(\overline{\phi(g)})$ and $\ell_{\gamma_k(N)}(z_{\phi(g)})\sim \ell(\phi(g))^k$.  It follows that \[\phi^n(g)=\phi^n(g)'\cdot\prod_{i=0}^{n-1}\phi^i(z_{\phi^{n-i}(g)}).\]  Here, $\phi^n(g)'$ has length $\sim(K_{\overline{\phi}})^n$, and each $z_{\phi^{n-i}(g)}$ has length $O(\ell(\overline{\phi^{n-i}(g)})^k)$ within $\gamma_k(N)$.  The justification behind this last estimate is that for any $\eps>0$ there is a $C>1$ and an $M$ such that for any $n,m>M$, \[\frac{1}{C}(K-\eps)^{km}\leq \ell(z_{\phi^m(g)})\leq C(K+\eps)^{km},\] and the size of the unit ball in $\gamma_k(N)$ is scaled by no more than $C(K+\eps)^{kn}$ upon the application of $\phi^n$.

Within $\gamma_k(N)$, $\phi^i(z_{\phi^{n-i}(g)})$ has length $\preceq (K_{\overline{\phi}})^{kn}$, so that in $N$, \[\prod_{i=0}^{n-1}\phi^i(z_{\phi^{n-i}(g)})\] has length $\preceq n^{1/k}(K_{\overline{\phi}})^{n}$.  It follows that $\ell(\phi^n(g))\preceq n\cdot (K_{\overline{\phi}})^n$.  By Lemma \ref{l:eps}, the result follows.
\end{proof}

Corollaries \ref{t:trivial} and \ref{t:uni} follow immediately from the fact that unipotent automorphisms of torsion--free abelian groups have spectral radii equal to one, and from Theorem \ref{t:pos}.

The remainder of the article is devoted to a proof of Theorem \ref{t:unipoly} and a discussion of its consequences.
We first would like to reduce the discussion to homologically trivial automorphisms of nilpotent groups.  Let $N$ be a finitely generated nilpotent group and let $\phi$ be an automorphism of $N$ which induces a unipotent automorphism of $H_1(N,\bZ)$.  We lose no information by assuming there is no torsion in $H_1(N,\bZ)$, since there is a map of the map $N\to T(H_1(N,\bZ))$ whose kernel has finite index in $N$ and is torsion free.  Here $T(H_1(N,\bZ))$ is the torsion subgroup of $H_1(N,\bZ)$, which can be made into a quotient by choosing a basis for the torsion--free part.  Proposition \ref{p:fi} shows that the entropy of $\phi$ as an automorphism of the kernel is equal to the entropy of $\phi$.

Consider the semidirect product \[1\to N\to N'\to \bZ\to 1,\] where the $\bZ$--conjugation action is given by $\phi$.  An easy calculation using the fact that $\phi-I$ induces a nilpotent endomorphism of $H_1(N,\bZ)$ shows that the lower central series of $N'$ terminates in finitely many steps, and the total number of terms in the lower central series is bounded above by the Hirsch rank of $N$ plus one (cf. \cite{R}).  It follows that $N'$ is nilpotent.  Furthermore, $N'$ contains a copy of $N$ together with the data of the action of $\phi$ acting on $N$.  Since the action of $\phi$ is given by conjugation by a group element $t\in N'$, we have that the action of $\phi$ extends to all of $N'$ in a way which acts trivially on $H_1(N',\bZ)$.

It is much easier to analyze the structure of homologically trivial automorphisms of nilpotent groups.  Let $N$ be a nilpotent group.  It is a standard fact that the commutator bracket is a bilinear function \[[\cdot,\cdot]:\gamma_i(N)/\gamma_{i+1}(N)\times\gamma_j(N)/\gamma_{j+1}(N)\to\gamma_{i+j}(N)/\gamma_{i+j+1}(N).\]  This fact is usually stated as ``the lower central series is a central filtration".  The following result is basic (cf. \cite{BL}):
\begin{lemma}
Let $\phi\in\Aut(N)$ and suppose that $\phi$ acts trivially on $H_1(N,\bZ)=N^{ab}$.  Then $\phi$ acts trivially on $\gamma_i(N)/\gamma_{i+1}(N)$ for all $i$.
\end{lemma}
\begin{proof}
The proof is by induction on commutator depth.  Suppose that $a\in\gamma_{i-1}(N)$ and $t\in N$ is arbitrary.  We may assume that $\phi(t)=t\cdot c$ where $c\in [N,N]$ and $\phi(a)=a\cdot c'$, where $c'\in\gamma_i(N)$.  The commutator $[t,a]$ lies in $\gamma_i(N)$, and it suffices to show that $\phi([t,a])$ differs from $[t,a]$ by an element of $\gamma_{i+1}(N)$.  We have \[\phi([t,a])=\phi(t)^{-1}\phi(a)^{-1}\phi(t)\phi(a)=c^{-1}t^{-1}c'^{-1}a^{-1}tcac'.\]  Switching $t^{-1}$ with $c'^{-1}$ perturbs $c'^{-1}$ by an element of $\gamma_{i+1}(N)$, since $c'\in\gamma_i(N)$.  So, there is a $c''\in\gamma_{i+1}(N)$ such that \[\phi([t,a])=c''c^{-1}t^{-1}a^{-1}tcac'.\]  Since $c\in [N,N]$ already, conjugating $t^{-1}a^{-1}t$ by $c$ perturbs it by an element of $\gamma_{i+1}(N)$.  The claim follows.
\end{proof}

It follows that if $\phi$ is a homologically trivial automorphism of any group $N$ and $g\in\gamma_i(N)$ then $\phi(g)=g\cdot z$, where $z\in\gamma_{i+1}(N)$.

\begin{proof}[Proof of Theorem \ref{t:unipoly}]
Suppose first that $\phi$ is a homologically trivial automorphism of $N$.  In this case, the proof is by induction on the length of the lower central series.  We will prove that for each $g\in N$, $\ell(\phi^n(g))$ is bounded by a polynomial in $n$ which depends on $g$ and the length of the lower central series of $N$.

Fix $g\in N$.  We may assume that $g\in N\setminus [N,N]$.  We write $\phi(g)=g\cdot c_2$, where $c_2\in [N,N]=\gamma_2(N)$.  Then $\phi(c_2)=c_2\cdot c_3$, where $c_3\in \gamma_3(N)$.  We continue to write $\phi(c_i)=c_ic_{i+1}$, where $c_{i+1}\in\gamma_{i+1}(N)$.  Eventually, we get $\phi(c_k)=c_k$ for some $k$, since $N$ is nilpotent.  Let $M$ be the maximum of the lengths of the $\{c_i\}$ which arise this way.  Note that $M$ depends on $g$, in general.

We now iterate $\phi$ and write $\phi^n(g)$ in terms of the result for $\phi^{n-1}(g)$ according to the rule $\phi(c_i)=c_ic_{i+1}$ for $i<k$.  We can estimate the length of $\phi^n(g)$ by counting the number of occurrences of each $c_i$.  Note that there are $n$ occurrences of $c_2$ in $\phi^n(g)$.  The number of occurrences of $c_3$ is approximately \[\sum_{i=1}^{n-1} i= O(n^2).\]  In general, the number of occurrences of $c_j$ is approximately \[\sum_{i=1}^{n-j+2} i^{j-2}=O(n^{j-1}).\]  If $N$ is $k$-step nilpotent, the length of $\ell(\phi^n(g))$ grows like $Mn^{k-1}$.

Now suppose that $\phi$ acts unipotently on $H_1(N,\bZ)$.  In this case, we construct the larger semidirect product $N'$ as above.  If $g\in N<N'$, we have that $\ell_{N'}(\phi^n(g))=O(Q(n))$ for some polynomial $Q$.  The inclusion of $N\to N'$ is not a quasi--isometry, so the word length of $\phi^n(g)$ as an element of $N$ can be vastly different from its length as an element of $N'$.  However, Osin's result (Lemma \ref{l:osin}) shows that the distortion is bounded by a polynomial $P$.  Composing $P$ and $Q$, we see that $\ell_N(\phi^n(g))=O(P\circ Q(n))$, the desired result.
\end{proof}

Consider word growth in $H_1(N,\bZ)$ under the action of a unipotent automorphism.  It is clear that $\phi$ is conjugate (over $\bR$, at least) to a matrix which is upper triangular and has ones down the diagonal.  Elementary observations analogous to those in the proof of Theorem \ref{t:unipoly} show that word growth in $H_1(N,\bZ)$ under iterations of $\phi$ is bounded by a polynomial of degree no larger than $\rk H_1(N,\bZ)-1$.  There is however a fundamental distinguishing feature of homologically trivial automorphisms:

\begin{prop}\label{l:trivial}
Let $N$ be a nilpotent group and $\phi\in \Aut(N)$ a homologically trivial automorphism.  Let $N'$ be the semidirect product of $N$ by $\bZ$, with the conjugation action of $\bZ$ given by $\phi$.  Then the lengths of the upper central series of $N$ and $N'$ coincide.
\end{prop}
\begin{proof}
Let $t$ and $N_0\subset N$ generate $N'$.  Clearly the Hirsch rank of $N'$ is no more than $\rk N+1$.  It is easy to see that $t$ and $N_0$ commute with the center of $N$, so that $Z(N)\subset Z(N')$.  It follows that $N'/Z(N)$ has rank no more than $\rk (N/Z(N))+1$.  Observe finally that $\rk H_1(N',\bZ)=\rk H_1(N,\bZ)+1$, so we obtain the conclusion by induction on the length of the upper central series of $N$.
\end{proof}

The assumption that $\phi$ act trivially on the homology of $N$ is essential since we can evidently construct nilpotent groups with arbitrarily long central series as iterated semidirect products with unipotent monodromies.  If $N$ is a finitely generated nilpotent group, the degree of polynomial growth of words inside of $N$ increases as the nilpotence degree increases (see \cite{dlH} and the references therein).  If we form a semidirect product \[1\to N\to N'\to\bZ\to 1,\] where the $\bZ$--conjugation action is given by an automorphism which acts unipotently but nontrivially on the homology of $N$, we expect the nilpotence degree of $N'$ to be larger than that of $N$, so that $\phi$ cannot induce linear growth rate of words in $N$.

We end with a question related to Theorem \ref{t:unipoly}:

\begin{quest}
Let $N$ be a finitely generated nilpotent group and let $\phi$ be an automorphism which induces a unipotent automorphism of $H_1(N,\bQ)$.  We have seen that for each $g\in N$, $\ell(\phi^n(g))=O(n^r)$ for some $r$.  Is it possible to bound $r$ as a function of $\rk H_1(N,\bQ)$?
\end{quest}

\end{document}